\documentclass[runningheads, a4paper]{llncs}        

\usepackage{amsmath,amsfonts,amssymb}   
\usepackage[english]{babel}
\usepackage{calrsfs}
\usepackage{hyperref}
\usepackage{multicol}
\usepackage{array}
\usepackage{tabularx}
\usepackage{stmaryrd}

\newcommand*\lo[1]{\mathcal{#1}}

\begin{document}
\title{The cumulative hierarchy in \\ Homotopy Type Theory}

\author{Ioannis Eleftheriadis\inst{1}\orcidID{0000-0003-4764-8894}}
\authorrunning{I. Eleftheriadis}
%
\institute{Mathematical Institute, University of Oxford 
\email{ioannis.eleftheriadis@univ.ox.ac.uk}\\
}
\maketitle

\begin{abstract}
We explore the cumulative hierarchy $V$ defined in Chapter 10 of the HoTT book. We begin by showing how to translate formulas of set theory in HoTT, and proceed to examine which axioms are satisfied in $V$. In particular, we show that $V$ models ZF$^-$ in HoTT+PR, while LEM is required to obtain full ZF. Finally, we attempt to model constructive set theories in V, and although this is achieved for ECST, we only obtain IZF and CZF with LEM.
\end{abstract}

\section{Introduction}

Homotopy Type Theory (HoTT) is an emerging field of study within mathematics and theoretical computer science, based on a connection between abstract homotopy theory and intensional type theory. This connection manifests itself in the form of a duality; on the one hand HoTT can be seen as an interpretation and extension of  Martin-Löf Type Theory (MLTT) using concepts and ideas from homotopy theory, while on the other it can be seen as an internal language for homotopy theory and higher category theory. Moreover, HoTT offers a homotopical framework for the foundation of mathematics, which differs from the standard set-theoretic foundations in many respects. Here, a central role is played by Voevodsky’s \emph{Univalence Axiom} (UA) which essentially allows isomorphic objects to be identified, something that is frequently carried out in informal mathematical practice. While it is believed that these univalent foundations will eventually become a viable alternative to set theory, it is not clear a priori how the two compare in terms of consistency strength. 

A large body of work in the subject is devoted to building models of (extensions of) HoTT within the current foundations. However, there seem to be few results on models of set theory within HoTT. This is partly due to the fact that this side of the question sits awkwardly between type theory and set theory, at times requiring expertise in the former and at other times in the latter. Furthermore, there is ambiguity as to how the underlying logic should be interpreted. Aczel's original model of set theory in MLTT (\cite{aczel}) used a propositions-as-types interpretation, which is strong enough to allow for full Constructive Zermelo-Fraenkel (CZF) set theory to be modelled. Other pieces of work treat the logic differently, resulting in models for different set theories.

In the homotopical setting, the main point of reference is the 10th chapter of \cite{hott}. There, a "cumulative hierarchy of sets" is constructed as a higher inductive type. Despite its name, the construction of this type does not reflect the way that the standard Von Neumann hierarchy is built. This cumulative hierarchy instead functions as a fixed hierarchy of small sets relative to some ambient type universe, much like \emph{Zermelo-Fraenkel algebras} do in algebraic set theory. Furthermore, the book implicitly uses a propositions-as-$(-1)$-types interpretation for the underlying logic, which although very strict, permits the possibility of consistently extending HoTT with a corresponding law of excluded middle. 

The purpose of this work is to clarify and partially extend Section 10.5 of the HoTT book. More precisely, we begin by explicitly providing semantics for relational languages in HoTT, using the propositions-as-$(-1)$-types interpretation. After reviewing the construction of $V$ and some of its properties, we proceed to translate axioms of set theory and explore which of these hold in $V$. For all of the axioms of ZF but Replacement and Separation, these translations fully correspond to the ones given in the book, while we provide and prove correspondences for the other two. Next, we show how under additional assumptions we may recover Powerset, Foundation, and Choice, whose proofs, although straightforward, do not appear to be in the literature. 

The final section is devoted to constructive set theories, and contains a combination of new results and old results translated to this homotopical setting. Amongst other things, we prove that $V$ is a model of \emph{Elementary Constructive Set Theory} (ECST), which merely rests on showing that $V$ models Strong Infinity, as $\Delta_0$-Separation is essentially covered in the book. We then briefly attempt to see if $H(x)$, the class of sets hereditarily an image of a set in $x$, is a set in $V$, answering this in the positive for $\omega$ as well as for general $x$ under LEM. Finally, we consider the constructive set theories CZF and IZF with their much-discussed Collection principles. Although we are unable to show that any of these holds in V, we observe that they may be derived in HoTT+LEM, and are therefore consistent with HoTT. 

We shall work in HoTT:= MLTT + UA + HITs with two type universes $U:U'$. Recall that a type $P:U$ is called a mere proposition (or (-1)-type) if $isProp(P):\equiv \prod_{x,y:P} x=y$ is inhabited, and $S:U$ is called an h-set (or 0-type) if $isSet(S):\equiv \prod_{x,y:S} isProp(x=y)$ is inhabited. We let $Prop_U :\equiv \sum_{P:U} isProp(P)$ and $Set_U:\equiv \sum_{S:U} isSet(S)$, and we often write $P:Prop_U$ to mean $P:U$ and $isProp(P)$ is inhabited, and likewise for $Set_U$. Also, $\exists$ and $\lor$ refer to the propositional truncation of $\Sigma$ and $+$, and "merely" denotes a propositional truncation. The following axioms may be consistently assumed:

\begin{itemize}
	\item Propositional Resizing,
	\begin{center}
PR $:\equiv  Prop_U \simeq Prop_{U'}$;
\end{center}	 

	\item Law of Excluded Middle,
	\begin{center}
 LEM$:\equiv \prod_{P:U}(isProp(P) \implies (P + \neg P))$;
\end{center}	

	\item Axiom of Choice,
\begin{center}
AC$:\equiv \prod_{S:U}\prod_{X:U}\prod_{Y:X\to U}(isSet(S) \land \prod_{x:X} isSet(Y(x))$\\$ \implies ((\prod_{x:X}\| Y(x)\|)\to \| \prod_{x:X}Y(x)\| ))$.
\end{center} 
\end{itemize}

Finally, we recall the following sets of axioms, expressed in the language of set theory (LST) $\lo L=\{\in\}$:
\begin{center}
$\textbf{ZF}^-$ = Extensionality + Empty Set + Pairing + Union + Infinity + \\ Powerset + Separation Scheme + Replacement Scheme 
\end{center}
\begin{center}
$\textbf{ZF}$ = $\textbf{ZF}^-$ + Foundation
\end{center}
\begin{center}
$\textbf{ZFC}$ = $\textbf{ZF}$ + Choice
\end{center}

\section{Homotopical semantics}

\begin{definition}
Let $\lo L$ be a finite relational language with equality. Assume that we have a tuple $\lo T:U$ consisting of a set $T:Set_U$ and functions $R_T: T^n \to Prop_U$ for each n-ary relation symbol $R$ of $\lo L$. For $\lo T$ as above, we define $\llbracket-\rrbracket_\lo T$ on formulas of $\lo L$ by induction on their structure:

\begin{itemize}
	\item $\llbracket x=y\rrbracket_\lo T :\equiv \lambda xy. x=_T y $
	\item $\llbracket R(x_1,...,x_k)\rrbracket_\lo T :\equiv R^n_T$, for all n-ary relation symbols
	\item $\llbracket\neg \phi\rrbracket_\lo T:\equiv \lambda \bar x. (\llbracket\phi\rrbracket_\lo T (\bar x) \to \textbf{0})$
	\item $\llbracket\phi \to \psi\rrbracket_\lo T:\equiv \lambda \bar x.(\llbracket\phi\rrbracket_\lo T (\bar x) \implies \llbracket\psi\rrbracket_\lo T (\bar x) )$
	\item $\llbracket\phi \land \psi\rrbracket_\lo T :\equiv \lambda \bar x.(\llbracket\phi\rrbracket_\lo T(\bar x) \land \llbracket\psi\rrbracket_\lo T(\bar x) ) $
	\item $\llbracket\phi \lor \psi\rrbracket_\lo T :\equiv \lambda \bar x.(\llbracket\phi\rrbracket_\lo T (\bar x) \lor \llbracket\psi\rrbracket_\lo T (\bar x))$
	\item $\llbracket\exists x \phi(x)\rrbracket_\lo T :\equiv \exists(x:T).\llbracket\phi\rrbracket_\lo T(x)$
	\item $\llbracket\forall x \phi(x)\rrbracket_\lo T :\equiv \forall(x:T).\llbracket\phi\rrbracket_\lo T(x) $
\end{itemize}

\end{definition}
We have slightly simplified the situation by assuming FV($\phi)=$ FV($\psi$), but it is easy to see how this definition works in general. It also follows that $\llbracket \phi \rrbracket_{\lo T} : T^m \to Prop_U$, where $m = \#$FV$(\phi)$. Although this translation could be defined for arbitrary types $T$, we would like $\llbracket\phi\rrbracket_\lo T$ to map into $Prop_U$, and so for all $x,y:T, x=_Ty$ must be a mere proposition, i.e. $T$ a set. 

\begin{definition}
Let $\lo L$ and $\lo T$ be as above. We call $\lo T$ a \emph{HoTT $\lo L$-structure}, and for $\phi \in Form(\lo L)$ such that $\llbracket\phi\rrbracket_\lo T : T^k \to Prop_U$ we write

\begin{center}
$\lo T \models \phi$ if and only if $\forall(t_1,...,t_k:T).\llbracket\phi\rrbracket_\lo T(t_1,...,t_k)$ is inhabited. 
\end{center}
For $S\subseteq Form(\lo L)$, we write $\lo T \models S$ if and only if $\lo T \models \phi$, for all $\phi \in S$.
\end{definition}

This translation provides sound semantics for Intuitionistic Predicate Logic (IPL).
\begin{theorem}[Soundness for IPL and HoTT]
Let $\lo L, \phi$ be as above and $S\subseteq Form(\lo L)$. Then: 
\begin{center}
$S \vdash_{IPL} \phi$ implies that for all HoTT $\lo L$-structures $\lo T$ such that $\lo T \models S$,  $\lo T \models \phi$.
\end{center}
\end{theorem}

\begin{proof}
We consider the axioms and rules of IPL given in \cite{rathjen}, section 2.5. Their translations in HoTT are trivially inhabited, and hence a proof in IPL produces a deduction in HoTT. 
\end{proof}

This establishes that we may internalise proofs in IPL as deductions in HoTT. Naturally, one may ask if there exists an analogous completeness theorem. A potential proof of this could invoke the completeness of IPL with respect to the Kripke semantics given in \cite{rathjen}, and a translation of Kripke structures in HoTT.

\section{V in HoTT}

Recall the definition of $V$ in section 10.5 of \cite{hott}. This is a type in $U'$ with terms:
\begin{itemize}
	\item \textbf{set}$: \prod_{A:U} (A\to V)\to V$
	\item \textbf{setext}$:\prod_{A,B:U} \prod_{f:A\to V} \prod_{g:B \to V} Eq\_img(f,g) \to \textbf{set}(A,f) = \textbf{set}(B,g)$
	\item \textbf{0-trunc}$:\prod_{x,y:V} \prod_{p,q:x=y} (p=q)$,
\end{itemize}
where 
\begin{center}
$Eq\_img(f,g) := (\forall (a:A).\exists(b:B).f(a)=g(b)) \land (\forall(b:B).\exists(a:A).f(a)=g(b))$
\end{center}

The induction principle for $V$ says that, given $P: V\to Set_U$, in order to construct $\phi: \prod_{x:V}P(x)$, it suffices to give the following.
\begin{itemize}
	\item For any $f:A\to V$, assume as given $\phi(a)$ for all $a:A$ and construct $\phi(set(A,f))$.
	\item For any $f:A\to V$ and $g:B\to V$ satisfying Eq$\_$img$(f,g)$, assume as given $\phi(f(a))$ and $\phi(g(b))$ for all $a:A, b:B$ and
\begin{center}
$\forall(a:A).\exists(b:B).\exists(p:f(a)=g(b)).\phi(f(a))=_p^P \phi(g(b)))$ \\ $\land \forall(b:B).\exists(a:A).\exists(p:f(a)=g(b)).\phi(f(a))=_p^P \phi(g(b)))$,
\end{center}
			and construct a dependent path $\phi($set$(A,f))=_{setext(A,B,f,g)}^P\phi($set$(B,g))$.
\end{itemize}

With this, we recover the $\in$ relation, applying induction to the second argument.
\begin{center}
$x \in set(A,f):\equiv (\exists(a:A).x=f(a))$,
\end{center}
Hence, $(V, \in)$ is an $\in$-structure in the sense of definition 1.2. We call any mere predicate $P:V \to Prop$ a \emph{class}, and we say that a class \emph{U-small} if its codomain is $Prop_U$. We also call $P$ a \emph{pure class} if it is the image of a $\in$-formula with one free variable, i.e. $P \equiv \llbracket \phi(x) \rrbracket_V : V \to Prop$. We may define a $U$-small resizing of the identity relation. 

\begin{definition}
Define the \textbf{bisimulation} relation $\sim : V \times V \to Prop_U$ using induction on both arguments, and letting
\begin{center}
$set(A,f)\sim set(B,g) :\equiv (\forall(a:A).\exists(b:B).f(a)\sim g(b)) \land$ \\$ (\forall(b:B).\exists(a:A).f(a)\sim g(b))$.
\end{center}
Similarly, we let $\tilde \in: V\times V \to Prop_U$ be $(x,set(A,f)) \mapsto \exists(a:A).x \sim f(a)$.
\end{definition}

\begin{lemma}
For any $u,v :V$ we have $(u=_V v) =_{U'} (u \sim v)$.
\end{lemma}

\begin{proof}
Lemma 10.5.5 in \cite{hott}.
\end{proof}

With this, one proves the following.

\begin{lemma}
For every $u:V$ there is a given $[u]:Set_U$ and a function $m_u : [u] \to V$ such that $u = set([u], m_u)$ and $m_u$ is monic, i.e. $m \circ f = m \circ g \implies f = g$. 
\end{lemma}

\begin{proof}
See \cite{hott}, Lemma 10.5.6, and \cite{ledent} section 3.2.
\end{proof}

Using the above propositions, we obtain the following.

\begin{proposition}
$(V, \in)\models$ Extensionality + Empty Set + Pairing + Union + Infinity + Exponentiation + $\in$-induction. Furthermore we have:
\begin{enumerate}
	\item H-Replacement: Given any $r:V \to V$ and $x:V$, there merely exists a $w:V$ such that $y \in w \iff \exists(z:V).z \in x \land y=r(z)$ for all $y:V$.
	\item U-Separation: Given any $a:V$ and $C:V\to Prop_U$, there merely exists a $w:V$ such that $x \in w \iff x \in a \land C(x)$ for all $x:V$.
\end{enumerate}
\end{proposition}

\begin{proof}
Theorem 10.5.8 in \cite{hott}.
\end{proof}
The last two differ from the usual Replacement and Separation. Indeed, the Replacement scheme says that for all $\phi(x,y)$ with two free variables, we have
\begin{center}
$\forall S(\forall x \exists ! y \phi(x,y)  \implies \exists w \forall v (v \in w \iff \exists u \in S \phi(u,v)))$.
\end{center}
We can see that $\phi(x,y)$ plays the role of a class function. It is not clear, however, that such a class function, which is represented by a formula, gives rise to an actual function $r:V\to V$, which is needed in the version of Replacement given above. Thankfully, this works out just fine.

\begin{proposition}
H-Replacement implies $\llbracket\chi\rrbracket_V$ for all instances of the Replacement scheme.
\end{proposition}
\begin{proof}
Let $\phi(x,y)$ be a formula with two variables in an instance of Replacement. We write $P:\equiv \llbracket \phi \rrbracket_V : V \times V \to Prop$, and fix $S:V$. We begin by noting that $\Theta(x) :\equiv \sum_{y:V} (P(x,y) \times \prod_{z:V} (P(x,z) \to z=y))$ is a mere proposition. Indeed, if $(y_1,r_1), (y_2,r_2) : \Theta$ then $r_1 = r_2$ as $P(x,y) \times \prod_{z:V} (P(x,z) \to z=y)$ is a mere proposition, and furthermore $pr_2(r_1)(y_2)(pr_1(r_2)) : y_1 = y_2$. Hence $(y_1,r_1) = (y_2,r_2)$. It therefore follows that $\exists !(y:V).P(x,y) \equiv \| \Theta \| \simeq \Theta$, and so we may obtain $f: \prod_{x:V} \Theta(x)$ and define $r:V \to V$ as $\lambda x.pr_1(f(x))$. Finally, H-Replacement implies that $\exists(w:V).\forall(v:V).(v \in w \iff \exists(u \in S).v = r(u))$, but $v = r(u)$ holds if and only if $P(u,v)$, and so the conclusion follows. 
\end{proof}

Essentially, the above proof boils down to the fact that mere propositions are closed under unique existential quantification, i.e. if $P : U \to Prop_U$ then $\exists ! (x:U). P(x) :\equiv \sum_{x:U}(P(x) \land \prod_{y:U} ( P(y) \to x = y ) )$ is a mere proposition.
\
Unfortunately, $U$-Separation does not fully correspond to the Separation scheme. We call a class $C: V \to Prop$ \emph{separable} if for any $a:V$ the class $a \cap C:\equiv \{x \mid x\in a \land C(x) \}$ is a set in $V$. $U$-Separation says that $U$-small classes are separable, whereas the translation of the Separation scheme requires that any pure class be separable. Since $V:U'$, we can see that $\llbracket\phi(x)\rrbracket_V : V \to Prop_{U'}$, and so we need $U'$-separability to obtain full Separation. Assuming PR, this is no longer an issue, and in fact, we may model ZF$^-$.

\begin{theorem}[HoTT + PR]
$(V,\in)\models$ ZF$^-$.
\end{theorem}

\begin{proof}
It remains to show that $\llbracket$Powerset$\rrbracket_V \equiv \forall(v:V).\exists(\lo P(v):V).\forall(x:V).(x\in \lo P(v) \iff x\subseteq v)$ is inhabited. Let $\Omega :\equiv Prop_{U_0}$, $v:V$, and take its monic presentation $[v]\rightarrowtail V$. We would like to define a map $r : (V\to \Omega) \to V$ such that $r(P) \equiv \{ \{m_v(t) \mid t:[v] \land P(m_v(t)) \} \mid P: V\to \Omega \}$. So, we let $[v]_P :\equiv \sum_{t:[v]} P(m_v(t))$, and take r:$\equiv P \mapsto$ set$([v]_P, m_v \circ pr_1)$. Finally let $\lo P(v) :\equiv$ set$(V\to \Omega, r)$. If $x \subseteq v$: Take $P\equiv \lambda z. z \in x : V \to \Omega$. Then $r(P) \in \lo P(v)$ and for all $y:V$ we have that, $y \in r(P) \iff \exists(a: [v]_P).(y = m_v\circ pr_1(a)) \iff \exists(b:[v])(y = m_v(b) \land m_v(b) \in x) \iff y \in x$. So $x = r(P)$ by Extensionality, and hence $x \in \lo P(v)$. If $x \in \lo P(v)$: Then there merely exists $P:V \to \Omega$ such that $x = r(P)$. Hence, for all $y:V$ we have that $y \in x \iff y \in r(P) \iff \exists(a:[v]_P).(y = m_u \circ pr_1(a)) \iff \exists(b:[v]).(y = m_u(b) \land P(m_u(b))) \implies y \in v$. Hence $x \subseteq v$.  

\end{proof}

By adding LEM to our type theory, we also obtain Foundation. 

\begin{theorem}[HoTT + LEM]
$(V, \in)\models$ ZF. 
\end{theorem}

\begin{proof}
We ought to show that $\llbracket$Foundation$\rrbracket_V \equiv \forall(v:V).(v \neq \emptyset \implies \exists(x:V).(x \in v \land x\cap v = \emptyset))$ is inhabited. Take $v:V$, and let $C:\equiv \lambda(x:V).\neg(x \in v) : V \to Prop$. We use the contrapositive of $\in$-induction with $C$, to obtain an inhabitant of $\exists(v:V).\neg C(v) \implies \exists(x:V).(\neg C(x) \land \forall(y:V).(y \in x \implies C(a)))$. Hence, if $v \neq \emptyset$, we may produce $x:V$ with $x \in v$ such that $\forall y \in x, y \not\in v$, i.e. $x \cap v = \emptyset$. 
\end{proof}

A reasonable question is whether we could get ZF with weaker assumptions. Unfortunately the following implies that LEM is necessary.

\begin{proposition}
$\llbracket Foundation\rrbracket_V$ implies $LEM_U$.
\end{proposition}

\begin{proof}
Let $p : Prop_U$, and $C: V \to Prop_U$ to be $C(x) :\equiv (x = 0 \land p) \lor x = 1$, where $0 = \emptyset, 1 = \{0\}$. Technically, this mere proposition lives in $U'$ because we have used equality, but it may be resized to $U$ by using the bisimulation relation. Using Infinity and $U$-Separation we form the set $ S_p :\equiv \{ x \in \omega \mid x=1 \lor (x=0 \land p) \} : V$. It follows that $1 \in S_p$, so $S_p$ is non-empty. Using Foundation, we obtain $z :V$ such that $z \in S_p$ and $z \cap S_p = \emptyset$, i.e. $\forall( y \in z).(y \not\in S_p)$. It follows by definition of $S_p$ that $ z = 1 \lor (z = 0 \land p)$. However, $z = 1 \implies 0 \not\in S_p \implies \neg p$, and $(z = 0 \land p) \implies p$. Hence $\neg p \lor p$.
\end{proof}

\begin{corollary}
$(V, \in) \models$ ZF implies LEM.
\end{corollary}

Finally, assuming AC we obtain a full model of ZFC

\begin{theorem}[HoTT + AC]
$(V,\in)\models$ ZFC.
\end{theorem}
\begin{proof}
\begin{center}
$\llbracket$Axiom of Choice$\rrbracket_V \equiv \forall(x:V).((\forall(y \in x).\exists(z:V).z \in y) \implies$ \\ $ \exists(c \in (\cup x)^x).\forall(y \in x).c(y) \in y)$.
\end{center}
So, let $x:V$ and take $A: V \to U'$ to be the constant family $y \mapsto V$, and $P: \prod_{x:V} A(x) \to U'$ to be the map $(y,z) \mapsto (y \in x \implies z \in y)$. Observe that $V$ is a set, $A(y)$ is a set for all $y:V$, and $P(y,z)$ is a mere proposition for all $y,z:V$. AC therefore gives us that $(\forall(y:V).\exists(z:V).y \in x \implies z \in y) \implies (\exists(g:\prod_{y:V}A(y)).\forall(y:V).y \in x \implies g(y)\in y)$. Hence, supposing that $(\forall(y \in x).\exists(z:V).z \in y)$, we obtain a $g:\prod_{y:V}A(y)\equiv V \to V$ as above. Thus, we define $g':V\to V$ by $y \mapsto \langle y,g(y)\rangle$ using Pairing. Using Replacement with $g'$ and $x$, we obtain $c:V$ containing exactly the ordered pairs $\langle y,g(y)\rangle$ for all $y \in x$. We can then see that $function(c, x, \cup x)$, as by the assumption $y \in x \implies g(y) \in y \subseteq \cup x$. Hence $c \in (\cup x)^x$ and $\forall y \in x, c(y)\equiv g(y) \in y$.
\end{proof}

Still, by assuming LEM or AC in our type theory we lose the computationally valuable constructive character of HoTT. Hence, we turn to constructive set theories, which are formulated in IPL. 

\section{Constructive Set Theories}
We let $\Sigma := $Extensionality + Empty Set + Pairing + Union. 
\begin{definition}
\textbf{ECST} $:= \Sigma$ + Strong Infinity + $\Delta_0$-Separation + Replacement, where
\begin{center}
Strong Infinity $:= \exists a(Ind(a) \land \forall b (Ind(b) \implies a \subseteq b)),$ \\ $Ind(x):=(\emptyset \in x \land \forall y \in x \ y \cup \{y\} \in x )$
\end{center}

\end{definition}

This is a good candidate since it omits the problematic axioms of ZF, namely Separation, Foundation, and Powerset. In fact, we shall prove that $(V,\in)$ is a model of ECST in HoTT. We begin by rewriting the proof that $V$ models $\Delta_0$-Separation found in \cite{hott} as Corollary 10.5.9.

\begin{proposition}
$(V,\in) \models \Delta_0$-Separation.
\end{proposition}

\begin{proof}
We call a class $\Delta_0$ if it is the image of a $\Delta_0$ formula of set theory. We shall show that $\Delta_0$-classes are $U$-small, and hence separable by Proposition 2.1.9. 

Let $\phi(x)$ be a $\Delta_0$ formula of set theory, and consider the class $C :\equiv \llbracket\phi(x)\rrbracket_V : V \to Prop_{U'}$. We define a new interpretation of formulas $\llbracket -\rrbracket_V^1$ by changing the definition of $\llbracket x=y\rrbracket_V^1$ to $\lambda xy.x\sim y$, and interpreting $\phi(x)$ in $(V, \tilde\in)$. By induction we have that $\forall(x:V).(\llbracket\psi(x)\rrbracket_V(x) \iff \llbracket\psi(x)\rrbracket_V^1(x))$ for all $\psi(x) \in Form(\in)$, and so $\forall(x:V).(C(x)\iff \tilde C(x))$. Hence $\forall(x:V)(C(x)=\tilde C(x))$ by univalence, and so $C = \tilde C$ by function extensionality. Hence, it suffices to show that $\tilde C$ is $U$-small. This shall be done in the metatheory by induction on the structure of $\phi(x)$. 

The base cases $x\sim y:\equiv \llbracket x=y\rrbracket_V^1$ and $x \ \tilde\in \ y:\equiv \llbracket x \in y\rrbracket_V^1$  are already $U$-small. Furthermore, $U$ is closed under the mere-propositional operations, so it remains to verify this for the bounded quantifiers. Indeed, we have that 

\begin{center}
$\llbracket\exists(x \in a)\chi(x)\rrbracket_V^1:\equiv \mid \mid \sum_{x:V}(x \  \tilde\in \ a \land \llbracket\chi\rrbracket_V^1(x)) \mid \mid$.
\end{center}

We know that $(x\  \tilde\in \ a \land \llbracket\chi\rrbracket_V^1(x))$ is $U$-small assuming by induction that $\llbracket\chi\rrbracket_V^1(x)$ is. Using the type of members of $a$ as per Lemma 2.2, we obtain that 

\begin{center}
$\sum_{x:V}(x \ \tilde\in \ a \land \llbracket\chi\rrbracket_V^1(x)) =_{U'} \sum_{x:[a]}\llbracket\chi\rrbracket_V^1(x)$, so 
\end{center}

\begin{center}
$\llbracket\exists(x \in a)\chi(x)\rrbracket_V^1 =_{U'} \mid \mid \sum_{x:[a]}P(x) \mid \mid$,
\end{center}

which is in $U$, as $U$ is closed under propositional truncation. For the universal case, we analogously obtain that $\llbracket\forall(y \in b)\chi(y)\rrbracket_V^1 = \prod_{x:[a]}\llbracket\chi\rrbracket_V^1$ which is also in $U$. 

\end{proof}

By tweaking the proof that Infinity holds in $V$, we also obtain Strong Infinity. Notice that Strong Infinity does not only stipulate the existence of an inductive set, but of a smallest such one. The witness of this axiom should clearly play the role of $\omega$. Under Foundation, this is of course equivalent to regular Infinity but this is not always true. In our case, we may obtain Strong Infinity via the induction principle of $\mathbb{N}$.

\begin{proposition}
$(V, \in) \models$ Strong Infinity.
\end{proposition}

\begin{proof}
The proof of (regular) Infinity is witnessed by $\omega = set(\mathbb{N},I)$, where $I: \mathbb{N} \to V$ is given by the recursion $I(0):\equiv \emptyset$, and $I(succ(n)):\equiv I(n) \cup \{I(n)\}$. Fix $b:V$ such that $Ind(b):\equiv \emptyset \in b \land \forall(x \in b).(x \cup \{x\} \in b)$ holds. Since $x\in \omega \iff \exists(n:\mathbb{N}).(x = I(n))$, it suffices to show that $\forall(n:\mathbb{N}).(I(n) \in \omega \implies I(n) \in b)$. Using the induction principle of $\mathbb{N}$, we have $I(0):\equiv \emptyset \in b$, and assuming $I(n) \in b$, we have that $I(succ(n)):\equiv I(n) \cup \{I(n)\}$, and so $I(succ(n)) \in b$ by the transitivity of $b$. It follows that $\forall(x \in \omega).(x \in \beta)$, and since $b$ was arbitrary, $\omega$ is the least transitive set.
\end{proof}

The above two propositions imply the following theorem which appears to be new in the literature. 

\begin{theorem}
HoTT proves that $(V,\in)$ is a model of ECST. 
\end{theorem}

Still, ECST does not fully capture the set-theory of $(V, \in)$, as it lacks Exponentiation and $\in$-induction. We may consider the following which is independent from ECST.

\begin{proposition}[ITER$_\omega$]
The Full Iteration Scheme, which says that for a class $C$, a class function $F(x,y)$, and $c \in C$, there is a set $\tilde C$ and a unique function $H : \omega \to \tilde C$ such that $H(\emptyset)=c_0$, and $F(H(x), H(x^+))$ for all $x \in \omega$, holds in V. 
\end{proposition} 

\begin{proof}
Take $C: V \to Prop$, $F: V \times V \to Prop$ such that $\forall(x:V).\exists !(y:V).F(x,y)$ and $c : V$ such that $C(c)$ holds. By proposition 2.2, we may internalise the class function $F$ to a function $r: V \to V$ with $F(x,r(x))$ true. Define $f : \mathbb{N} \to V$ by $h(0) \equiv c, f(succ(n)) \equiv r(f(n))$, and take $\tilde C :\equiv set(\mathbb{N}, f)$. Likewise, define $h: \mathbb{N} \to V$ by $n \mapsto \langle I(n), f(n)\rangle$, where $I: \mathbb{N} \to V$ is the map such that $\omega = set(\mathbb{N}, I)$. Taking $H :\equiv set(\mathbb{N}, h)$, we see that $H$ is the unique set-theoretic function such that $F(H(x), H(x^+))$ for all $x \in \omega$.
\end{proof}

Using this with the class function $x \cup \bigcup x$ we obtain the following.

\begin{corollary}
$(V, \in) \models \forall x (x$ has a transitive closure).
\end{corollary}

One may ask if for every set $x: V$ we may construct the set $H(x):V$ of sets hereditarily an image of a set in $x$, i.e. the smallest set such that whenever $b \in x$ and $f: b \to H(x)$, $Im f \in H(x)$. By clever use of ITER$_\omega$ we can see that $H(\omega)$ is a set in $V$. The general case is equivalent to the following axiom, which is independent from all the ones we have examined thus far. 

\begin{definition}[fREA]
The functional Regular Extension Axiom is the proposition:
\begin{center}
$\forall x \exists C ( x \subseteq C \land trans(C) \land \forall b \in C \ \forall f \in C^b Im f \in C)$
\end{center}
\end{definition}

This cannot be proven with an iteration argument as there might be "unbounded" functions $f: b \to x$ for $b \in x$. Nonetheless it does hold in ZF, but the proof is non-trivial and makes heavy use of Foundation, and hence of LEM. 

\begin{theorem}
ZF $\vdash$ fREA.
\end{theorem}

\begin{proof}
Proposition 4.1 in \cite{frea}.
\end{proof}

\begin{corollary}[HoTT+LEM]
$(V, \in)\models$ fREA, and hence $H(x):V$ is a set for all $x:V$. 
\end{corollary}

It is still unclear if fREA can be proved without LEM, but it seems unlikely as it is independent from most constructive set theories. Also, ITER$_\omega$ may be obtained from ECST+$\in$-induction, and so it is more appropriate to consider set theories which contain $\in$-induction. 

\begin{definition}
\textbf{IZF} $:= \Sigma$ + Infinity + Powerset + Separation + $\in$-induction + Collection, where

\begin{center}
Collection $:=\forall a(\forall x \in a \ \exists y \ \theta(x,y) \implies \exists b \forall x \in a \ \exists y \in b \ \theta(x,y)), b \not\in FV(\theta)$.
\end{center}

\end{definition}

\begin{definition}
\textbf{CZF} $:= \textbf{ECST}$ - Replacement + $\in$-induction + Strong Collection + Subset Collection, where

\begin{center}
Strong Collection $:=$\\$ \forall a(\forall x \in a \ \exists y \phi(x,y) \implies \exists b(\forall x \in a \ \exists y \in b \phi(x,y) \land \forall y\in b \ \exists x\in a \ \phi(x,y)))$
\end{center}
\begin{center}
Subset Collection $:= \forall a \forall b(\exists c\ \forall u(\forall x \in a \ \exists y \in b \psi(x,y,u)  \implies $\\$\exists d \in c (\forall x \in a \ \exists y \in d \ \psi(x,y,u) \land \forall y \in d \ \exists x \in a \ \psi(x,y,u))))$,
\end{center}
where the formulas $\phi$ and $\psi$ may have any number of other free variables.
\end{definition}

In contrast to IZF, CZF is amenable to methods from ordinal-theoretic proof theory. Let us see how CZF relates to ECST.

\begin{proposition}
The following implications hold:
\begin{enumerate}
	\item Strong Collection $\vdash_{IPL}$ Collection $\land$ Replacement
	\item ECST $\vdash_{IPL}$ Subset Collection $\implies$ Exponentiation
\end{enumerate}
\end{proposition}
\begin{proof}
1 is straighforward. For 2, see Theorem 5.1.2 in \cite{rathjen}
\end{proof}
It therefore follows that CZF is strictly stronger than ECST. From a classical perspective, CZF has exactly the same strength as ZF. We write CZF$^c$ to emphasize that we mean CZF with classical logic.
\begin{proposition}
ZF and CZF$^c$ prove the same theorems.
\end{proposition}

\begin{proof}
Corollary 4.2.8 in \cite{rathjen}.
\end{proof}

With these, one may conclude the following:

\begin{corollary}[HoTT + LEM]
$(V, \in) \models$ CZF + IZF.
\end{corollary}

This observation, although not useful in a constructive context, establishes that all three Collection axioms are consistent with HoTT. In fact, it appears that these are independent from HoTT; since they are independent in IPL from the set of axioms that were shown to hold in HoTT (even in HoTT + PR), their satisfaction in $V$ purely depends on the "outer" properties of $V$, i.e. its constructors and induction principle. These, however, only appear to reflect properties of $Set_U$, where again, these axiom do not seem to hold as is. To show their independence, one would have to construct a model of HoTT where these axioms fail in its internal cumulative hierarchy, something that goes much beyond the purposes of this work.

\end{document}